\documentclass[preprint,3p]{elsarticle}

\usepackage{amssymb}
\usepackage{amsmath}
\usepackage{amsthm}
\usepackage{color}
\usepackage{enumitem}
\usepackage{tikz}
\usepackage{mathrsfs}
\usepackage{stmaryrd}
\usepackage{appendix}
\usepackage[colorlinks=true]{hyperref}
\usepackage[hyphenbreaks]{breakurl}

\theoremstyle{plain}
  \newtheorem{thm}{Theorem}[section]
  \newtheorem{lem}[thm]{Lemma}
  \newtheorem{prop}[thm]{Proposition}
  \newtheorem{cor}[thm]{Corollary}
\theoremstyle{definition}
  \newtheorem{defn}[thm]{Definition}
  \newtheorem{exmp}[thm]{Example}
  \newtheorem{rem}[thm]{Remark}

\DeclareMathAlphabet{\mathcal}{OMS}{cmsy}{m}{n}

\DeclareMathOperator{\Int}{Int}
\DeclareMathOperator{\Ext}{Ext}

\makeatletter
\def\ps@pprintTitle{%
 \let\@oddhead\@empty
 \let\@evenhead\@empty
 \def\@oddfoot{\centerline{\thepage}}%
 \let\@evenfoot\@oddfoot}
\makeatother

\def\rto{\mathrel{\relbar\joinrel\joinrel\joinrel\relbar\joinrel\joinrel\mapstochar\joinrel\joinrel\relbar\joinrel\joinrel\joinrel\rightarrow}}

\newcommand{\da}{\downarrow}
\newcommand{\ua}{\uparrow}
\newcommand{\ra}{\rightarrow}

\newcommand{\lda}{\swarrow}
\newcommand{\rda}{\searrow}

\newcommand{\bv}{\bigvee}

\renewcommand{\phi}{\varphi}

\newcommand{\be}{\beta}

\newcommand{\CC}{\mathcal{C}}

\newcommand{\CQ}{\mathcal{Q}}

\newcommand{\sC}{\mathsf{C}}

\newcommand{\sM}{\mathsf{M}}

\newcommand{\sP}{\mathsf{P}}

\newcommand{\sj}{\mathsf{j}}

\newcommand{\Fix}{\mathsf{Fix}}

\newcommand{\bbQ}{\mathbb{Q}}

\newcommand{\Bond}{\mathbf{Bond}}
\newcommand{\Bb}{\mathbf{b}}
\newcommand{\Bc}{\mathbf{c}}
\newcommand{\BB}{\mathbf{B}}

\newcommand{\FB}{\mathfrak{B}}

\newcommand{\SGrp}{\mathbf{SGrp}}

\newcommand{\Quant}{\mathbf{Quant}}

\newcommand{\Rel}{\mathbf{Rel}}

\newcommand{\Sup}{\mathbf{Sup}}

\newcommand{\ResRelMult}{\mathbf{ResRel}_{\otimes}}

\newcommand{\dR}{R^{\da}}
\newcommand{\uR}{R^{\ua}}

\newcommand{\dS}{S^{\da}}
\newcommand{\uS}{S^{\ua}}
\newcommand{\dT}{T^{\da}}

\newcommand{\dphi}{\phi^{\da}}
\newcommand{\uphi}{\phi^{\ua}}

\newcommand{\op}{{\rm op}}

\newcommand{\PX}{\sP X}
\newcommand{\PY}{\sP Y}

\newcommand{\MR}{\sM R}
\newcommand{\MS}{\sM S}

\newcommand{\CX}{\sC X}
\newcommand{\CY}{\sC Y}

\renewcommand{\leq}{\leqslant}
\renewcommand{\geq}{\geqslant}

\renewcommand{\to}{\longrightarrow}

\numberwithin{equation}{section}

\allowdisplaybreaks

\begin{document}

\begin{frontmatter}

\title{Quantales as formal concept lattices}

\author[S]{Lili Shen}
\ead{shenlili@scu.edu.cn}

\author[Z]{Xiaojuan Zhao\corref{cor}}
\ead{xjzhao@swjtu.edu.cn}

\cortext[cor]{Corresponding author.}
\address[S]{School of Mathematics, Sichuan University, Chengdu 610064, China}
\address[Z]{School of Mathematics, Southwest Jiaotong University, Chengdu 611756, China}

\begin{abstract}
We establish a new bridge connecting quantales, semigroups and the theory of formal concept analysis. By introducing residuated relations whose domains are semigroups, we show that every quantale arises as the formal concept lattice induced by such a relation. Furthermore, this construction yields an equivalence between the category of quantales and a category of such residuated relations whose morphisms are multiplicative bonds.
\end{abstract}


\end{frontmatter}

\section{Introduction}

\emph{Quantales} \cite{Mulvey1986,Rosenthal1990} are complete lattices
equipped with a multiplication distributing over arbitrary suprema.
They provide a flexible ordered-algebraic framework in which completeness
and associative multiplication interact, and have been studied from a
variety of algebraic, geometric, and topological perspectives \cite{Resende2007,Resende2018,Resende2018a,Weiss2019,Wang2023,Mesablishvili2026}.  \emph{Formal concept analysis} (FCA), on the other hand, starts from a relation between an object set and an attribute set (called a \emph{formal context}), and forms a complete lattice through its Galois closure (called the \emph{formal concept lattice}) \cite{Ganter1999,Davey2002}. The two constructions meet naturally only after one answers a structural question: if the object set is additionally equipped with a semigroup structure, when does the induced multiplication on its subsets descend through the Galois closure?

This paper answers that question at the level of relations. A \emph{residuated relation}
\[
R\colon X\rto Y
\]
with $X$ a semigroup and $Y$ a set allows an attribute of a product to be transported to a residual attribute of either factor (Definition \ref{residuated-relation-def}). We prove that this elementary condition is strong enough to make the closure
\[
\sj_R=\dR\uR\colon\PX\longrightarrow\PX
\]
a quantic nucleus. Consequently its fixed points, equivalently the extents of
the context $R$, form a quantale. We then prove that every quantale arises
in this way, and characterize the residuated relations presenting a given quantale $Q$ (Theorem \ref{sup-inf-dense-criterion}). Thus quantales admit relation-level presentations more flexible
than the canonical order presentation: the object set carries a semigroup
structure, while residual attributes record how its multiplication interacts
with the incidence relation.

The object-level construction has a natural morphism-level lift. Ordinary \emph{bonds}
\cite{Ganter2007} between formal contexts form a category $\Bond$ equivalent to
the category $\Sup$ of complete lattices and sup-preserving maps \cite{Mori2008}; in the quantaloid setting this is the \emph{back diagonal}
construction of \cite{Shen2016a}. We show that, for residuated relations, the condition
\[
\Bb[xx']=\Bb[x]\otimes_S\Bb[x']
\]
selects exactly those bonds whose induced maps preserve the descended multiplication. The resulting category
\[
\ResRelMult
\]
of residuated relations and multiplicative bonds is equivalent to $\Quant$ (Theorem \ref{main-equivalence-theorem}).

The general Galois and monoidal frameworks used here are already established:
Resende's theory of sup-lattice $2$-forms encompasses the Galois quotient
associated with a formal context \cite{Resende2004}, while Mori established the
lattice tensor on bonds and the resulting strong monoidal equivalence with
$\Sup$; see \cite{Mori2008} and \cite[Theorem~27]{Tull2020}. The contribution
of the present paper is the relation-level algebra added to these frameworks:
the residual condition on a semigroup domain forces the Galois closure to be a
quantic nucleus, and the multiplicative row condition characterizes exactly
the bonds inducing quantale homomorphisms.

Moreover, Section~\ref{section-semigroup-objects} gives a categorical interpretation
of the main theorem in terms of semigroup objects in $\Bond$, while Section~\ref{section-pseudogroup-completions}
gives a focused application: the enveloping quantal frame constructed in
\cite{Resende2007} from an abstract complete pseudogroup is obtained as a
concept quantale.

\section{Quantales} 
\label{Quantales}

A \emph{quantale} \cite{Rosenthal1990} is a complete lattice $Q$ equipped with an associative binary operation
$*\colon Q\times Q\to Q$ such that the multiplication preserves arbitrary suprema in each variable:
\begin{equation} \label{quantale-multiplication-preserves-suprema}
\Big(\bv\limits_{i\in I}p_i\Big)*q=\bv\limits_{i\in I}(p_i*q)
\quad\text{and}\quad
p*\Big(\bv\limits_{i\in I}q_i\Big)=\bv\limits_{i\in I}(p*q_i)
\end{equation}
for all $p,q\in Q$ and all families $\{p_i\}_{i\in I}$ and $\{q_i\}_{i\in I}$
in $Q$. A \emph{homomorphism} of quantales is a map
$f\colon P\to Q$ which preserves arbitrary suprema and multiplication:
\[
f\Big(\bv\limits_{i\in I}p_i\Big)=\bv\limits_{i\in I}f(p_i),
\qquad
f(p*p')=f(p)*f(p')
\]
for all $p,p'\in P$ and all families $\{p_i\}_{i\in I}$ in $P$. We denote by 
\[
\Quant
\] 
the category of quantales and their homomorphisms.

\begin{rem}
Unless otherwise stated, quantales in this paper are not assumed to be unital. The unital version is obtained by requiring quantales to have a multiplicative unit and homomorphisms to preserve it.
\end{rem}

Since the multiplication in a quantale preserves arbitrary suprema in each variable, for any $p,q\in Q$ the maps
\[
-*q\colon Q\to Q\quad\text{and}\quad p*-\colon Q\to Q
\]
both admit right adjoints. We write these right adjoints as
\[
-\lda q\colon Q\to Q
\quad\text{and}\quad
p\rda -\colon Q\to Q,
\]
respectively. Thus,
\begin{equation} \label{quantale-lda-rda}
p*q\leq r\iff p\leq r\lda q\iff q\leq p\rda r
\end{equation}
for all $p,q,r\in Q$. Explicitly,
\[
r\lda q=\bv\{p\in Q\mid p*q\leq r\},
\qquad
p\rda r=\bv\{q\in Q\mid p*q\leq r\}.
\]
These residual operations will be used later to construct canonical examples of residuated relations.

We write the product of a \emph{semigroup} \cite{Clifford1961,Howie1995} $X$ by juxtaposition: thus the product of
$x,x'\in X$ is $xx'$.  The symbol $*$ is reserved for the
multiplication of a quantale.  Every semigroup $X$ gives rise to a quantale
$\PX$, whose underlying complete lattice is the powerset of $X$ and whose
multiplication is defined by
\[
A*B=\{ab\mid a\in A,\ b\in B\}
\]
for all $A,B\subseteq X$.  Thus $(\PX,*)$ is the \emph{free quantale}
generated by $X$:

\begin{prop} \label{PX-free-quantale}
The assignment $X\mapsto(\PX,*)$ defines a left adjoint to the forgetful functor 
\[
\Quant\to\SGrp,
\] 
where $\SGrp$ is the category of semigroups and their homomorphisms.
\end{prop}

This is the non-unital analogue of Rosenthal's free-quantale adjunction
\cite[Proposition~2.3.2]{Rosenthal1990}; the same proof applies with units
omitted.



A \emph{quantic nucleus} \cite{Rosenthal1990} on a quantale $Q$ is a map $\sj\colon Q\to Q$ such that
\begin{enumerate}[label=(N\arabic*)]
\item \label{nucleus:m} $x\leq y\implies\sj(x)\leq\sj(y)$,
\item \label{nucleus:c} $x\leq\sj(x)$,
\item \label{nucleus:i} $\sj\sj(x)=\sj(x)$,
\item \label{nucleus:*} $\sj(x)*\sj(y)\leq\sj(x*y)$
\end{enumerate}
for all $x,y\in Q$. Thus a quantic nucleus is a closure operator compatible with the quantale multiplication. As an immediate consequence of \ref{nucleus:m}--\ref{nucleus:*}, we have
\[
\sj(x*y)\leq\sj(\sj(x)*\sj(y))\leq\sj\sj(x*y)=\sj(x*y),
\]
and consequently
\begin{equation} \label{j-jx-jy}
\sj(x*y)=\sj(\sj(x)*\sj(y))
\end{equation}
for all $x,y\in Q$.

Each quantic nucleus $\sj$ on $Q$ induces a quotient quantale, denoted by $Q_{\sj}$. Its underlying set is the set of fixed points
\[
Q_{\sj}=\{x\in Q\mid \sj(x)=x\}.
\]
For $x,y,x_i\in Q_{\sj}$ $(i\in I)$, the multiplication and suprema in $Q_{\sj}$ are given by
\begin{equation} \label{Qj-multiplication-sup}
x*_{\sj}y=\sj(x*y),
\qquad
\bigsqcup\limits_{i\in I}x_i=\sj\Big(\bv\limits_{i\in I}x_i\Big),
\end{equation}
where $\bigsqcup$ denotes the supremum in $Q_{\sj}$ and $\bv$ denotes the supremum in $Q$. Moreover, $\sj\colon Q\to Q_{\sj}$ is a homomorphism of quantales.

The following standard representation theorem says that every quantale is a quotient of a free quantale by a quantic nucleus.

\begin{thm} \label{Q-iso-PXj}
(See \cite[Theorem 3.1.2]{Rosenthal1990}.)
For every quantale $Q$, there exist a semigroup $X$ and a quantic nucleus $\sj$ on $\PX$ such that
\[
Q\cong(\PX)_{\sj}
\]
as quantales.
\end{thm}

\begin{rem}
The preceding theorem will be refined in this paper. We shall show that the nuclei relevant for representing quantales can be obtained from Galois closures induced by residuated relations. More precisely, a residuated relation $R\colon X\rto Y$ will induce the
quantic nucleus
\[
\sj_R:=\dR\uR\colon \PX\to\PX,
\]
and the corresponding quotient quantale will be precisely the formal concept lattice of $R$.
\end{rem}

\section{Formal concept lattices}

Recall that a \emph{relation} $R\colon X\rto Y$ between sets is a subset $R\subseteq X\times Y$, where
$(x,y)\in R$ is also written as $xRy$. Each relation $R\colon X\rto Y$ induces two maps
\[
\uR\colon\PX\to\PY
\quad\text{and}\quad
\dR\colon\PY\to\PX
\]
between the powersets of $X$ and $Y$, defined by
\begin{equation} \label{uR-dR-def}
\uR(A)=\{y\in Y\mid \forall a\in A,\ aRy\},
\qquad
\dR(B)=\{x\in X\mid \forall b\in B,\ xRb\}.
\end{equation}
Then $\uR$ and $\dR$ form a Galois connection between $\PX$ and $(\PY)^{\op}$:
\begin{equation} \label{uR-dv-dR}
B\subseteq \uR(A)\iff A\subseteq\dR(B)
\end{equation}
for all $A\subseteq X$ and $B\subseteq Y$. Consequently,
\begin{itemize}
\item both $\uR$ and $\dR$ send unions into intersections, and
\item both $\dR\uR\colon\PX\to\PX$ and $\uR\dR\colon\PY\to\PY$ are closure operators.
\end{itemize}

\begin{rem}
The Galois connection above is standard in formal concept analysis. It also
admits a formulation in the theory of sup-lattice $2$-forms: the relation
$R$ determines a $2$-form on the powersets of $X$ and $Y$, and the associated
orthogonality construction recovers the same Galois correspondence and concept
lattice; see \cite{Resende2004}. We use the ordinary FCA notation throughout.
\end{rem}

\begin{rem}
More generally, every distributor between enriched categories induces an
\emph{Isbell adjunction}. If sets $X$ and $Y$ are regarded as discrete
categories enriched in the two-element Boolean algebra
$\mathbf{2}=\{0,1\}$, then $R$ is a $\mathbf{2}$-distributor between them,
and the Isbell adjunction induced by $R$ specializes to the Galois connection
$\uR\dashv\dR$ between $\PX$ and $(\PY)^{\op}$; see
\cite{Lawvere1986,Day2007,Kelly2005,Stubbe2005,Shen2013a,Shen2014,Lai2017}.
\end{rem}

For $x\in X$ and $y\in Y$, write
\[
R[x]=\{y'\in Y\mid xRy'\}
\qquad\text{and}\qquad
R^{-1}[y]=\{x'\in X\mid x'Ry\}
\]
for the row of $x$ and the column of $y$, respectively. Thus
\[
\uR(\{x\})=R[x],
\qquad
\dR(\{y\})=R^{-1}[y].
\]
More generally,
\[
\uR(A)=\bigcap_{x\in A}R[x],
\qquad
\dR(B)=\bigcap_{y\in B}R^{-1}[y]
\]
for all $A\subseteq X$ and $B\subseteq Y$.

\begin{lem} \label{dRuR-fix}
Let $R\colon X\rto Y$ be a relation between sets. For each $A\subseteq X$ and $x\in X$, the following statements are equivalent:
\begin{enumerate}[label={\rm(\roman*)}]
\item $x\in\dR\uR(A)$.
\item For each $y\in Y$, if $aRy$ for all $a\in A$, then $xRy$.
\end{enumerate}
\end{lem}

\begin{proof}
By definition, $x\in\dR\uR(A)$ if and only if $xRy$ for every $y\in\uR(A)$. Since $y\in\uR(A)$ precisely means that $aRy$ for all $a\in A$, the assertion follows.
\end{proof}

In the theory of \emph{formal concept analysis} \cite{Ganter1999,Davey2002}, a \emph{formal context} (or \emph{context} for short) is a triple $(X,Y,R)$, where $X$ is the set of \emph{objects}, $Y$ is the set of \emph{attributes}, and $R\colon X\rto Y$ is a relation. The statement $xRy$ is read as ``the object $x$ has the attribute $y$''.

A \emph{formal concept} of a formal context $(X,Y,R)$ is a pair $(A,B)$, where $A\subseteq X$ and $B\subseteq Y$, such that
\[
A=\dR(B)
\quad\text{and}\quad
B=\uR(A).
\]
In this case, $A$ is called the \emph{extent} of the concept and $B$ is called its \emph{intent}. The set of all formal concepts is denoted by $\FB(X,Y,R)$ (or $\FB R$ for short), which is ordered by
\[
(A,B)\leq(A',B')
\iff
A\subseteq A'
\iff
B\supseteq B'.
\]
With this order, $\FB R$ is a complete lattice, called the \emph{formal concept lattice} (or \emph{concept lattice} for short) of $R$.

A subset $A\subseteq X$ is called an \emph{$R$-extent} if $A=\dR\uR(A)$, and a subset $B\subseteq Y$ is called an \emph{$R$-intent} if $B=\uR\dR(B)$. We write
\[
\begin{aligned}
&\Ext(R)=\Fix(\dR\uR)=\{A\subseteq X\mid A=\dR\uR(A)\}\quad\text{and}\\
&\Int(R)=\Fix(\uR\dR)=\{B\subseteq Y\mid B=\uR\dR(B)\}.
\end{aligned}  
\]
The assignments
\[
(A,B)\mapsto A
\quad\text{and}\quad
(A,B)\mapsto B
\]
identify $\FB R$ with $\Ext(R)$ and with $\Int(R)^{\op}$, respectively. In what follows, we shall mostly use the extent representation and write
\[
\MR=\Ext(R)=\Fix(\dR\uR).
\]
Thus $\MR$ is a complete lattice whose order is inclusion. Its infima are computed as intersections:
\[
\bigsqcap_{i\in I}A_i=\bigcap_{i\in I}A_i,
\]
and its suprema are computed by closing unions:
\[
\bigsqcup_{i\in I}A_i=\dR\uR\Big(\bigcup_{i\in I}A_i\Big),
\]
where $A_i\in\MR$ for all $i\in I$. Similarly, $\Int(R)$, ordered by inclusion, is a complete lattice: its infima are intersections, and its suprema are obtained by closing unions under $\uR\dR$. Moreover, the restrictions 
\begin{equation} \label{uR-dR-MR-IntR-iso}
\left.\uR\right|_{\MR}\colon\MR\to\Int(R)^{\op}\quad\text{and}\quad\left.\dR\right|_{\Int(R)}\colon\Int(R)^{\op}\to\MR
\end{equation}
are mutually inverse order isomorphisms.

We shall also need the standard notion of a bond between formal contexts \cite{Ganter2007}. Let
\[
R\colon X\rto Y
\quad\text{and}\quad
S\colon Z\rto W
\]
be two relations. A \emph{bond} from $R$ to $S$ is a relation $\Bb\colon X\rto W$ such that
\begin{enumerate}[label={\rm(B\arabic*)}]
\item \label{bond-def:row} for every $x\in X$, the row $\Bb[x]=\{w\in W\mid x\Bb w\}$ is an $S$-intent;
\item \label{bond-def:column} for every $w\in W$, the column $\Bb^{-1}[w]=\{x\in X\mid x\Bb w\}$ is an $R$-extent.
\end{enumerate}
Here rows are subsets of the attribute set $W$ of the second context, while columns are subsets of the object set $X$ of the first context.

Every bond $\Bb\colon R\to S$ induces a map
\begin{equation} \label{bond-induced-map}
\Bb_*\colon \MR\to\MS,\quad \Bb_*(A)=\dS\Bb^\ua(A),
\end{equation}
where $\Bb^\ua$ is defined by \eqref{uR-dR-def}, i.e.,
\[
\Bb^\ua(A)=\{w\in W\mid \forall a\in A,\ a\Bb w\}.
\]
Since $\Bb^\ua(A)\subseteq W$, the set $\dS\Bb^\ua(A)$ is an $S$-extent, and therefore $\Bb_*(A)\in\MS$.

The following standard fact is the well-known correspondence between bonds of formal contexts and sup-preserving maps between their concept lattices; see, for example, \cite{Ganter2007,Mori2008}. It is also the relational instance of the quantaloid-theoretic bond construction in \cite{Shen2016a}. We recall its elementary verification in the present notation.

\begin{prop} \label{bond-map-sup}
For every bond $\Bb\colon R\to S$, the induced map $\Bb_*\colon\MR\to\MS$ preserves arbitrary suprema.
\end{prop}

\begin{proof}
Let $A_i\in\MR$ for $i\in I$. Since $\Bb^\ua$ sends unions to intersections, we have
\[
\Bb^\ua\Big(\bigcup_{i\in I}A_i\Big)=\bigcap_{i\in I}\Bb^\ua(A_i).
\]
Moreover, the columns of $\Bb$ are $R$-extents, hence replacing a subset $A\subseteq X$ by its closure $\dR\uR(A)$ does not change $\Bb^\ua(A)$. Indeed, for every $w\in W$,
\[
w\in\Bb^\ua(A)
\iff
A\subseteq \Bb^{-1}[w],
\]
and $\Bb^{-1}[w]$ is an $R$-extent.

Therefore,
\[
\Bb_*\Big(\bigsqcup_{i\in I}A_i\Big)=
\dS\Bb^\ua\dR\uR\Bigl(\bigcup_{i\in I}A_i\Bigr)=
\dS\Bb^\ua\Bigl(\bigcup_{i\in I}A_i\Bigr)=
\dS\Big(\bigcap_{i\in I}\Bb^\ua(A_i)\Big).
\]
Finally, each $\Bb^\ua(A_i)$ is an $S$-intent because $\Bb$ has $S$-intent rows and intents are closed under arbitrary intersections. Hence $\dS$ sends the above intersection of intents to the supremum of the corresponding $S$-extents in $\sM S$ (see \eqref{uR-dR-MR-IntR-iso}), giving
\[
\dS\Big(\bigcap_{i\in I}\Bb^\ua(A_i)\Big)
=
\bigsqcup_{i\in I}\dS\Bb^\ua(A_i)
=
\bigsqcup_{i\in I}\Bb_*(A_i).
\]
Thus $\Bb_*$ preserves arbitrary suprema.
\end{proof}

Conversely, every sup-preserving map between concept lattices is induced by a unique bond. More precisely, let $h\colon\MR\to\MS$ be a sup-preserving map. Define a relation $\Bb_h\colon X\rto W$ by (cf. \eqref{uR-dv-dR})
\begin{equation} \label{bond-from-join-map}
x\Bb_h w
\iff
w\in\uS h\dR\uR(\{x\})
\iff
h\dR\uR(\{x\})\subseteq \dS(\{w\}).
\end{equation}

\begin{prop} \label{sup-map-bond}
Let $h\colon\MR\to\MS$ be a sup-preserving map. Then $\Bb_h$ is a bond from $R$ to $S$, and $(\Bb_h)_*=h$.
\end{prop}

\begin{proof}
First, for each $w\in W$, the column $\Bb_h^{-1}[w]$ is given by
\[
\Bb_h^{-1}[w]=\{x\in X\mid h\dR\uR(\{x\})\subseteq \dS(\{w\})\}.
\]
Since $h$ preserves arbitrary suprema, this set is an $R$-extent. Indeed, if $A=\dR\uR(\Bb_h^{-1}[w])$, then
\[
h(A)=h\Big(\bigsqcup_{x\in \Bb_h^{-1}[w]}\dR\uR(\{x\})\Big)
=\bigsqcup_{x\in \Bb_h^{-1}[w]}h\dR\uR(\{x\})
\subseteq \dS(\{w\}).
\]
If $x\in A$, then $\dR\uR(\{x\})\subseteq A$ (because $A$ is an $R$-extent). As every sup-preserving map is monotone, it follows that
\[
h\dR\uR(\{x\})\subseteq h(A)\subseteq \dS(\{w\}).
\]
Hence $x\in\Bb_h^{-1}[w]$. Therefore $A\subseteq \Bb_h^{-1}[w]$, and so $\Bb_h^{-1}[w]$ is an $R$-extent.

Next, for each $x\in X$, the row $\Bb_h[x]$ is an $S$-intent because
\[
\Bb_h[x]=\uS h\dR\uR(\{x\}).
\]
Thus $\Bb_h$ is a bond.

Finally, for every $A\in\MR$,
\[
\begin{aligned}
(\Bb_h)_*(A)
&=
\dS\Bb_h^\ua(A)
\\
&=
\dS(\{w\in W\mid \forall a\in A,\ a \Bb_h w\})
\\
&=
\dS\Big(\bigcap\limits_{a\in A}\uS h\dR\uR(\{a\})\Big)\\
&=
\dS\uS\Big(\bigsqcup_{a\in A}h\dR\uR(\{a\})\Big)
\\
&=
\bigsqcup_{a\in A}h\dR\uR(\{a\})
\\
&=
h\Big(\bigsqcup_{a\in A}\dR\uR(\{a\})\Big)
\\
&=
h(A).
\end{aligned}
\]
This proves $(\Bb_h)_*=h$.
\end{proof}

\begin{cor} \label{bonds-sup-maps}
For relations $R\colon X\rto Y$ and $S\colon Z\rto W$, bonds from $R$ to $S$ are in one-to-one correspondence with sup-preserving maps $\MR\to\MS$.
\end{cor}

\begin{proof}
By Propositions \ref{bond-map-sup} and \ref{sup-map-bond}, each bond $\Bb$ induces a sup-preserving map $\Bb_*$, and each sup-preserving map $h$ induces a bond $\Bb_h$ with $(\Bb_h)_*=h$.


It remains only to observe uniqueness. Let $\Bb$ be a bond, and fix $x\in X$ and $w\in W$. Put
\[
C_x=\dR\uR(\{x\}).
\]
Since $\Bb^{-1}[w]$ is an $R$-extent and $C_x$ is the least $R$-extent containing $x$, we have
\[
x\Bb w
\iff
C_x\subseteq\Bb^{-1}[w]
\iff
w\in\Bb^\ua(C_x).
\]
Moreover, $\Bb^\ua(C_x)$ is an $S$-intent, being the intersection of the $S$-intent rows $\Bb[a]$ for $a\in C_x$. Hence
\[
w\in\Bb^\ua(C_x)
\iff
w\in\uS\dS\Bb^\ua(C_x)
\iff
\dS\Bb^\ua(C_x)\subseteq\dS(\{w\}).
\]
By the definition of $\Bb_*$, this gives
\[
x\Bb w
\iff
\Bb_*\dR\uR(\{x\})\subseteq\dS(\{w\}).
\]
Hence $\Bb$ is recovered from $\Bb_*$ by formula \eqref{bond-from-join-map}.
\end{proof}

The preceding correspondence also supplies the standard category of bonds. Let
\[
\Bond
\]
have formal contexts as objects and bonds as morphisms. The identity morphism
on a context
\[
R\colon X\rto Y
\]
is the relation $R$ itself: its rows are the $R$-intents
$\uR(\{x\})$, its columns are the $R$-extents $\dR(\{y\})$, and
\[
R_*=1_{\MR}.
\]
For bonds
\[
\Bb\colon R\to S
\qquad\text{and}\qquad
\Bc\colon S\to T,
\]
where
\[
R\colon X\rto Y,
\qquad
S\colon Z\rto W,
\qquad
T\colon U\rto V,
\]
define their composite $\Bc\circ\Bb\colon R\to T$ to be the unique bond
corresponding under Corollary \ref{bonds-sup-maps} to the sup-preserving map
\[
\Bc_*\circ\Bb_*\colon\MR\to\sM T.
\]
Since $\Bb^\ua\dR\uR(\{x\})=\Bb[x]$, one has
\[
\Bb_*\dR\uR(\{x\})=\dS(\Bb[x]).
\]
Moreover, the recovery formula established in the proof of Corollary
\ref{bonds-sup-maps}, together with preservation of suprema, gives
\[
\Bc_*(A)\subseteq\dT(\{v\})
\iff
A\subseteq\Bc^{-1}[v]
\]
for every $A\in\MS$ and $v\in V$. It follows that the composite has the
elementwise description
\[
x(\Bc\circ\Bb)v
\iff
\dS(\Bb[x])\subseteq\Bc^{-1}[v]
\]
for all $x\in X$ and $v\in V$, and by construction
\[
(\Bc\circ\Bb)_*=\Bc_*\circ\Bb_*.
\]
Since bonds are uniquely determined by their induced maps, associativity and
the unit laws follow from the corresponding laws for composition of maps.
Thus $\Bond$ is a category.

\begin{rem} \label{Bond-BRel}
The category $\Bond$ has a quantaloid-theoretic interpretation. The notion of
a bond above is the specialization to the quantaloid $\Rel$ (see
\cite{Rosenthal1996} for the theory of \emph{quantaloids}) of the notion of a bond
between arrows in an arbitrary quantaloid \cite{Shen2016a}. More generally,
one begins with Chu connections and forms their quotient quantaloid of back
diagonals. For $\Rel$, the resulting quantaloid is denoted by $\BB(\Rel)$, and
it is canonically isomorphic to $\Bond$
(see \cite[Proposition~2.3.5]{Shen2016a}). Under this identification, a back
diagonal from $R\colon X\rto Y$ to $S\colon Z\rto W$ is exactly a relation
$\Bb\colon X\rto W$ whose rows are $S$-intents and whose columns are
$R$-extents. Thus the present definition is the elementwise form of the
back-diagonal construction.
\end{rem}

Consider the concept-lattice assignment
\[
\sM_0\colon\Bond\to\Sup,
\qquad
R\mapsto\MR,
\qquad
\Bb\mapsto\Bb_*.
\]
Corollary \ref{bonds-sup-maps} says that this functor is fully faithful. The
basic theorem on concept lattices (see \cite[Theorem 3]{Ganter1999} and
\cite[Theorem 3.9]{Davey2002}) says that every complete lattice is
isomorphic to the concept lattice of a formal context, so $\sM_0$ is
essentially surjective. Hence it is an equivalence of categories:

\begin{prop} (See \cite{Mori2008,Shen2016a}.) \label{bond-category-sup-equivalence}
The functor $\sM_0\colon\Bond\to\Sup$ is an equivalence of categories.
\end{prop}

More precisely, this is the bond formulation of Mori's Galois equivalence
\cite[Theorem~73]{Mori2008}; in the quantaloid setting, it is the case
$\CQ=\mathbf{2}$ of \cite[Proposition~2.3.5 and
Theorem~3.4.4]{Shen2016a}.

\section{Quantales as formal concept lattices}

We now explain how quantales arise as formal concept lattices of residuated
relations. The key point is that, for a semigroup $X$ and a set $Y$, a relation
\[
R\colon X\rto Y
\]
induces a Galois closure
\[
\sj_R:=\dR\uR\colon \PX\to\PX,
\]
and, under a residuation condition, this closure is a quantic nucleus on the free quantale $\PX$.

\begin{defn} \label{residuated-relation-def}
Let $X$ be a semigroup and $Y$ a set. A relation $R\colon X\rto Y$ is
\emph{residuated} if there exist maps
\[
\lda\colon Y\times X\to Y
\quad\text{and}\quad
\rda\colon X\times Y\to Y
\]
such that
\begin{equation} \label{residuated-relation-iff}
(xx',y)\in R\iff(x,y\lda x')\in R\iff (x',x\rda y)\in R
\end{equation}
for all $x,x'\in X$ and $y\in Y$.
\end{defn}

Equivalently, writing $xRy$ for $(x,y)\in R$, the condition says that the semigroup multiplication on the domain $X$ admits right residuals with respect to the relation $R$:
\[
(xx')Ry
\iff
xR(y\lda x')
\iff
x'R(x\rda y).
\]
Thus $y\lda x'$ and $x\rda y$ play the role of the right residuals of $y$ by $x'$ and of $y$ by $x$, respectively, but only relative to the relation $R$.

A basic special case is obtained from ordered semigroups. By a \emph{residuated ordered semigroup} we mean an ordered semigroup $X$ whose order relation
\[
\leq\ \colon X\rto X
\]
is a residuated relation. Explicitly, this means that there exist maps
\[
/\colon X\times X\to X
\quad\text{and}\quad
\backslash\colon X\times X\to X
\]
such that
\[
xx'\leq y
\iff
x\leq y/ x'
\iff
x'\leq x\backslash y
\]
for all $x,x',y\in X$. In this case, the residuals are internal to the ordered semigroup $X$.


In particular, as an immediate consequence of the residual adjunctions
\eqref{quantale-lda-rda}, every quantale $Q$ is a residuated ordered
semigroup. Example~\ref{example-sup-inf-dense-ordered-presentations}
extends this basic case from $Q$ itself to sup-dense and inf-dense residuated
ordered subsemigroups $A\subseteq Q$. Further concrete realizations of
residuated relations are collected in
Appendix~\ref{appendix-examples}.

\begin{prop} \label{dRuR-nucleus}
Let $R\colon X\rto Y$ be a residuated relation. Then $\sj_R$ is a quantic nucleus on the quantale $\PX$.
\end{prop}

\begin{proof}
The map $\sj_R=\dR\uR\colon\PX\to\PX$ is already a closure operator, since it is induced by the Galois connection between $\PX$ and $(\PY)^{\op}$. It remains to show that it is compatible with the quantale multiplication on $\PX$ (see \ref{nucleus:*}); that is,
\[
(\dR\uR(A))*(\dR\uR(B))\subseteq\dR\uR(A*B)
\]
for all $A,B\subseteq X$.

Let $x\in\dR\uR(A)$ and $x'\in\dR\uR(B)$. We show that $xx'\in\dR\uR(A*B)$. By Lemma \ref{dRuR-fix}, it suffices to prove that, for each $y\in Y$, if $(ab)Ry$ for all $a\in A$ and $b\in B$, then $(xx')Ry$. Suppose that $(ab)Ry$ for all $a\in A$ and $b\in B$. For every $b\in B$, residuation gives
\[
(ab)Ry
\iff
aR(y\lda b).
\]
Hence $aR(y\lda b)$ for all $a\in A$. Since $x\in\dR\uR(A)$, Lemma \ref{dRuR-fix} gives
$xR(y\lda b)$. Again by residuation, this is equivalent to
$bR(x\rda y)$. Thus $bR(x\rda y)$ for every $b\in B$. Since $x'\in\dR\uR(B)$, Lemma \ref{dRuR-fix} gives
$x'R(x\rda y)$. By residuation once more, this is equivalent to $(xx')Ry$. Therefore $xx'\in\dR\uR(A*B)$, as required.
\end{proof}

Since $\sj_R$ is a quantic nucleus on $\PX$, the quotient-quantale construction recalled in Section \ref{Quantales} applies. Its fixed-point lattice is precisely the lattice of $R$-extents:
\[
(\PX)_{\sj_R}
=
\Fix(\sj_R)
=
\MR.
\]
Thus the formulas in \eqref{Qj-multiplication-sup} immediately yield the following canonical quantale structure on the concept lattice of $R$.

\begin{prop} \label{MR-is-quantale}
Let $R\colon X\rto Y$ be a residuated relation. Then $\MR$ is a quantale, with multiplication and suprema given by
\[
A\otimes_R B
=
\dR\uR(A*B),
\qquad
\bigsqcup_{i\in I}A_i
=
\dR\uR\Big(\bigcup_{i\in I}A_i\Big)
\]
for all $A,B\in\MR$ and all families $\{A_i\}_{i\in I}$ in $\MR$. Moreover, the closure map
\[
\sj_R=\dR\uR\colon \PX\to\MR
\]
is a homomorphism of quantales. When the relation $R$ is clear from the context, we simply write $A\otimes B$ instead of $A\otimes_R B$.
\end{prop}

We now obtain the object-level representation theorem. Proposition~\ref{MR-is-quantale} shows that every residuated relation $R\colon X\rto Y$ gives rise to a quantale $\MR$. The following theorem proves the converse: every quantale is isomorphic to $\MR$ for some residuated relation $R\colon X\rto Y$. More precisely, it characterizes the residuated relations presenting a given quantale. It is a quantale-theoretic refinement of the basic theorem on concept lattices (see \cite[Theorem 3]{Ganter1999} and \cite[Theorem 3.9]{Davey2002}).

Let $Q$ be a complete lattice. A map $f\colon X\to Q$ is called \emph{sup-dense} if
\begin{equation} \label{sup-dense-def}
q=\bv\{f(x)\mid x\in X,\ f(x)\leq q\}
\end{equation}
for all $q\in Q$, and it is called \emph{inf-dense} if
\begin{equation} \label{inf-dense-def}
q=\bigwedge\{f(x)\mid x\in X,\ q\leq f(x)\}
\end{equation}
for all $q\in Q$.

\begin{thm} \label{sup-inf-dense-criterion}
Let $Q$ be a quantale, and let $R\colon X\rto Y$ be a residuated relation. Then $Q$ is isomorphic to $\MR$ as a quantale if and only if there exist a sup-dense semigroup homomorphism $f\colon X\to Q$ and an inf-dense map $g\colon Y\to Q$ such that
\begin{equation} \label{dense-representation-relation}
xRy
\iff
f(x)\leq g(y)
\end{equation}
for all $x\in X$ and $y\in Y$. In particular, every quantale $Q$ is isomorphic, as a quantale, to $\sM(Q,Q,\leq)$.
\end{thm}

\begin{proof}
{\bf Sufficiency.} Suppose first that such maps $f$ and $g$ exist. Define a map
\[
\Phi\colon \PX\to Q,\quad\Phi(A)=\bv f[A]
=
\bv\{f(a)\mid a\in A\}.
\]
We show that the restriction of $\Phi$ to $\MR$ is an isomorphism of
quantales.

First, $\left.\Phi\right|_{\MR}\colon\MR\to Q$ is injective. For
$A\subseteq X$ and $y\in Y$, condition
\eqref{dense-representation-relation} gives
\begin{align*}
y\in\uR(A)
&\iff
aRy\ \text{for all }a\in A
\\
&\iff
f(a)\leq g(y)\ \text{for all }a\in A
\\
&\iff
\bv f[A]\leq g(y).
\end{align*}
Hence
\begin{equation} \label{uR-A-Phi-A-gy}
\uR(A)
=
\{y\in Y\mid \Phi(A)\leq g(y)\}.
\end{equation}
For $A\in\MR$, we claim that
\[
A=\{x\in X\mid f(x)\leq \Phi(A)\},
\]
hence an extent $A$ is completely determined by $\Phi(A)$, which necessarily forces $\left.\Phi\right|_{\MR}$ to be injective. Indeed, if $x\in A$, then clearly $f(x)\leq \Phi(A)$. Conversely, suppose that $f(x)\leq \Phi(A)$. To show $x\in A$, since $A=\dR\uR(A)$, it is enough to show that $xRy$ for every $y\in\uR(A)$. But for every $y\in\uR(A)$, we have $\Phi(A)\leq g(y)$ by \eqref{uR-A-Phi-A-gy}, and consequently
\[
f(x)\leq \Phi(A)\leq g(y).
\]
Thus $xRy$, and therefore $x\in A$.

Second, $\left.\Phi\right|_{\MR}\colon\MR\to Q$ is surjective. For every $q\in Q$, define
\[
A_q=\{x\in X\mid f(x)\leq q\}.
\]
We show that $A_q$ is an $R$-extent. Since $g$ is inf-dense, combining \eqref{inf-dense-def} and \eqref{dense-representation-relation} we have
\begin{equation} \label{Aq-def}
\begin{aligned}
A_q
&=\{x\in X\mid f(x)\leq g(y)\ \text{for every }y\in Y\text{ with }q\leq g(y)\}\\
&=\{x\in X\mid xRy\ \text{for every }y\in Y\text{ with }q\leq g(y)\}\\
&=\dR(\{y\in Y\mid q\leq g(y)\}).
\end{aligned}
\end{equation}
Thus $A_q$ is an $R$-extent. Moreover, since $f$ is sup-dense, it follows from \eqref{sup-dense-def} that
\begin{equation} \label{Phi-Aq=q}
\Phi(A_q)
=
\bv\{f(x)\mid f(x)\leq q\}
=
q.
\end{equation}
Therefore, the restriction of $\Phi$ to $\MR$ is surjective.

Third, $\left.\Phi\right|_{\MR}\colon\MR\to Q$ preserves arbitrary suprema. Note that for every $U\subseteq X$, from \eqref{uR-A-Phi-A-gy}, \eqref{Aq-def} and \eqref{Phi-Aq=q} we obtain
\begin{equation} \label{Phi-dR-uR-U-Phi-Aq}
\Phi\dR\uR(U)
=
\Phi\dR(\{y\in Y\mid\Phi(U)\leq g(y)\})
=
\Phi(A_{\Phi(U)})=\Phi(U).
\end{equation}
Thus, for any family $\{A_i\}_{i\in I}$ in $\MR$,
\begin{align*}
\Phi\Big(\bigsqcup_{i\in I}A_i\Big)
&=
\Phi\dR\uR\Big(\bigcup_{i\in I}A_i\Big)
\\
&=
\Phi\Big(\bigcup_{i\in I}A_i\Big)
\\
&=
\bv f\Big[\bigcup_{i\in I}A_i\Big]
\\
&=
\bv_{i\in I}\bv f[A_i]
\\
&=
\bv_{i\in I}\Phi(A_i).
\end{align*}

Finally, $\left.\Phi\right|_{\MR}\colon\MR\to Q$ preserves multiplication. Since $f$ is a semigroup homomorphism, we have
\begin{align*}
\Phi(A\otimes_R B)
&=
\Phi\dR\uR(A*B)\\
&=
\Phi(A*B) & (\text{by \eqref{Phi-dR-uR-U-Phi-Aq}})\\
&=
\bv f[A*B]\\
&=
\bv\{f(ab)\mid a\in A,\ b\in B\}\\
&=
\bv\{f(a)*f(b)\mid a\in A,\ b\in B\}\\
&=
\Big(\bv f[A]\Big)*\Big(\bv f[B]\Big) & (\text{by \eqref{quantale-multiplication-preserves-suprema}})\\
&=
\Phi(A)*\Phi(B)
\end{align*}
for all $A,B\in\MR$, as desired.

{\bf Necessity.} Suppose that
\[
\Psi\colon \MR\to Q
\]
is an isomorphism of quantales. Define maps
\[
f\colon X\to Q,\quad f(x)=\Psi\dR\uR(\{x\})
\quad\text{and}\quad
g\colon Y\to Q,\quad g(y)=\Psi\dR(\{y\}).
\]
Then $f$ is a semigroup homomorphism, because
\begin{align*}
f(xx')
&=
\Psi\dR\uR(\{xx'\})\\
&=
\Psi\dR\uR(\{x\}*\{x'\})\\
&=
\Psi\dR\uR(\dR\uR(\{x\})*\dR\uR(\{x'\})) &(\text{by \eqref{j-jx-jy} and Proposition \ref{dRuR-nucleus}})\\
&=
\Psi(\dR\uR(\{x\})\otimes_R\dR\uR(\{x'\}))\\
&=
\Psi\dR\uR(\{x\})*\Psi\dR\uR(\{x'\})\\
&=
f(x)*f(x')
\end{align*}
for all $x,x'\in X$.

For every extent $A\in\MR$ and every $x\in X$, one has
\[
f(x)\leq\Psi(A)
\iff
\dR\uR(\{x\})\subseteq A
\iff
x\in A.
\]
Since
\[
A
=
\dR\uR\Big(\bigcup_{x\in A}\dR\uR(\{x\})\Big)
=
\bigsqcup_{x\in A}\dR\uR(\{x\}),
\]
it follows that
\[
\Psi(A)
=
\bv_{x\in A}f(x)
=
\bv\{f(x)\mid f(x)\leq\Psi(A)\}.
\]
Thus $f$ is sup-dense.

Similarly, for every extent $A\in\MR$ and every $y\in Y$, one has
\[
\Psi(A)\leq g(y)
\iff
A\subseteq\dR(\{y\})
\iff
y\in\uR(A).
\]
Since (cf. Lemma \ref{dRuR-fix})
\[
A=\bigcap_{y\in\uR(A)}\dR(\{y\}),
\]
it follows that
\[
\Psi(A)
=
\bigwedge_{y\in\uR(A)}g(y)
=
\bigwedge\{g(y)\mid\Psi(A)\leq g(y)\}.
\]
Thus $g$ is inf-dense.

Finally, for $x\in X$ and $y\in Y$,
\begin{align*}
xRy
&\iff
x\in\dR(\{y\})
\\
&\iff
\dR\uR(\{x\})\subseteq\dR(\{y\})
\\
&\iff
\Psi\dR\uR(\{x\})\leq\Psi\dR(\{y\})
\\
&\iff
f(x)\leq g(y).
\end{align*}
This proves the converse implication. For the final assertion, take $X=Y=Q$, $R={\leq}$, and $f=g=1_Q$. As observed above, the order relation ${\leq}\colon Q\rto Q$ is residuated, and $1_Q$ is both sup-dense and inf-dense. Hence $Q\cong\sM(Q,Q,\leq)$ as quantales.
\end{proof}

\begin{rem}
The map $g$ in Theorem~\ref{sup-inf-dense-criterion} is not required to
preserve any algebraic structure; indeed, the attribute set $Y$ in
Definition~\ref{residuated-relation-def} is only a set. The role of $g$ is to
express the incidence relation by
\[
xRy
\iff
f(x)\leq g(y)
\]
and to be inf-dense. By contrast, the homomorphism property of $f$ is used to
show that $\Phi(A)=\bv f[A]$ preserves the induced quantale multiplication.
In the canonical order presentation, both $f$ and $g$ are the identity map
$1_Q$.
\end{rem}

\begin{rem}
Theorem~\ref{sup-inf-dense-criterion} and the representation theorem of
Brown and Gurr \cite{Brown1993} both give relational representations of
arbitrary quantales, but the relations play different roles. Brown and Gurr
realize a quantale as a family of binary relations on a set, with
multiplication given by relational composition. Here a single relation
$R\colon X\rto Y$ serves as a formal context: its Galois closure yields the
underlying complete lattice, while multiplication is induced from the
semigroup structure on $X$.
\end{rem}

\begin{rem} \label{category-equivalence-interface}
The preceding representation theorem concerns objects. It will later be
lifted to a categorical statement. Namely, suitable bonds between residuated
relations induce sup-preserving maps between the corresponding concept
quantales; imposing compatibility with the descended multiplications singles
out precisely the quantale homomorphisms. This is the morphism-level
completion of the object-level representation.
\end{rem}

\section{Multiplicative bonds}

Let $R\colon X\rto Y$ and $S\colon Z\rto W$
be residuated relations. By Proposition \ref{MR-is-quantale}, the concept lattices
\[
\MR
\quad\text{and}\quad
\MS
\]
carry canonical quantale structures. We now give the morphism-level lift of
the preceding object-level representation by characterizing those bonds whose
induced maps preserve the quantale multiplication.

This is a refinement of Corollary \ref{bonds-sup-maps}. There we showed that, for arbitrary formal contexts, bonds are precisely the relational data inducing sup-preserving maps between concept lattices. In the present setting, the relations are residuated and the concept lattices carry quantale structures. Hence the natural question is which bonds correspond, under the correspondence of Corollary \ref{bonds-sup-maps}, to homomorphisms of quantales. Proposition \ref{multiplicative-bond-equivalence} answers this question by adding exactly one intrinsic compatibility condition, namely multiplicativity of rows.

The mutually inverse order isomorphisms in
\eqref{uR-dR-MR-IntR-iso} transport the extent-side multiplication of
$\MS$, given by Proposition~\ref{MR-is-quantale}, to the intent side.
Thus, for $P,Q\in\Int(S)$, write
\begin{equation} \label{P-otimes_S-Q-def}
P\otimes_S Q
=
\uS(\dS(P)*\dS(Q))
=
\uS\bigl(\dS(P)\otimes_S\dS(Q)\bigr).
\end{equation}
In the second expression, the inner multiplication $\otimes_S$ is the
extent-side multiplication of $\MS$. Then
\[
\dS(P\otimes_S Q)
=
\dS(P)\otimes_S\dS(Q),
\qquad
\uS(C\otimes_S D)
=
\uS(C)\otimes_S\uS(D)
\]
for all $P,Q\in\Int(S)$ and $C,D\in\MS$.

Let $\Bb\colon R\to S$ be a bond. For each $x\in X$, its $x$-th row
\[
\be_{\Bb}(x):=\Bb[x]=\{w\in W\mid x\Bb w\}
\]
is an $S$-intent (see \ref{bond-def:row}). Thus we have a map
\[
\be_{\Bb}\colon X\to\Int(S),
\qquad
x\mapsto \Bb[x].
\]

\begin{defn} \label{multiplicative-bond-def}
Let $R\colon X\rto Y$ and $S\colon Z\rto W$ be residuated relations. A bond
\[
\Bb\colon R\to S
\]
is called \emph{multiplicative} if
\[
\Bb[xx']
=
\Bb[x]\otimes_S\Bb[x']
\]
for all $x,x'\in X$.
\end{defn}

Equivalently, a bond $\Bb\colon R\to S$ is multiplicative if and only if its row map
\[
\be_{\Bb}\colon X\to\Int(S)
\]
is a semigroup homomorphism from $X$ to the semigroup of $S$-intents equipped with the multiplication $\otimes_S$.

We next prove that this intrinsic row-wise condition is precisely the condition that the induced map between concept lattices preserves the quantale multiplication.

\begin{lem} \label{bond-row-induced-map}
Let $\Bb\colon R\to S$ be a bond. Then, for every $A\in\MR$,
\[
\uS\Bb_*(A)=\Bb^\ua(A).
\]
Moreover, for every $x\in X$,
\[
\uS\Bb_*\dR\uR(\{x\})=\Bb[x].
\]
\end{lem}

\begin{proof}
Since $\Bb$ is a bond, each row $\Bb[x]$ is an $S$-intent. Thus, for every $A\subseteq X$,
\[
\Bb^\ua(A)=\bigcap_{a\in A}\Bb[a]
\]
is an $S$-intent. Then, it follows from \eqref{bond-induced-map} that
\[
\uS\Bb_*(A)
=
\uS\dS\Bb^\ua(A)
=
\Bb^\ua(A).
\]

For the second assertion, we use the fact that the columns of $\Bb$ are $R$-extents. Note that for every $w\in W$,
\[
w\in \Bb^\ua\dR\uR(\{x\})
\iff
\dR\uR(\{x\})\subseteq \Bb^{-1}[w].
\]
Since $\Bb^{-1}[w]$ is an $R$-extent, this is equivalent to $x\in \Bb^{-1}[w]$, 
that is, to $x\Bb w$. Therefore
\[
\Bb^\ua\dR\uR(\{x\})=\Bb[x].
\]
Together with the first assertion, this gives
\[
\uS\Bb_*\dR\uR(\{x\})=\Bb[x]. \qedhere
\]
\end{proof}

\begin{prop} \label{multiplicative-bond-equivalence}
Let
\[
\Bb\colon R\to S
\]
be a bond between residuated relations. Then $\Bb$ is multiplicative if and only if the induced map
\[
\Bb_*\colon \MR\to\MS
\]
preserves quantale multiplication; that is,
\[
\Bb_*(A\otimes_R B)
=
\Bb_*(A)\otimes_S\Bb_*(B)
\]
for all $A,B\in\MR$.
\end{prop}

\begin{proof}
{\bf Necessity.} Suppose that $\Bb$ is multiplicative. 
Let $x,x'\in X$. Then
\begin{align*}
&\uS\Bb_*\left(\dR\uR(\{x\})\otimes_R\dR\uR(\{x'\})\right)\\
={}&
\uS\Bb_*\dR\uR(\{xx'\})
&(\text{by \eqref{j-jx-jy} and Propositions \ref{dRuR-nucleus}, \ref{MR-is-quantale}})\\
={}& 
\Bb[xx'] 
&(\text{by Lemma \ref{bond-row-induced-map}})\\
={}& 
\Bb[x]\otimes_S\Bb[x']
&(\text{by Definition \ref{multiplicative-bond-def}})\\
={}&
\uS\Bb_*\dR\uR(\{x\})\otimes_S\uS\Bb_*\dR\uR(\{x'\})
&(\text{by Lemma \ref{bond-row-induced-map}})\\
={}&
\uS\left(\Bb_*\dR\uR(\{x\})\otimes_S\Bb_*\dR\uR(\{x'\})\right).
&(\text{by \eqref{P-otimes_S-Q-def}})
\end{align*}
Since $\uS$ is injective on $S$-extents, we get
\begin{equation} \label{Bb-preserves-multiplication-x-x'}
\Bb_*\left(\dR\uR(\{x\})\otimes_R\dR\uR(\{x'\})\right)
=
\Bb_*\dR\uR(\{x\})\otimes_S\Bb_*\dR\uR(\{x'\}).
\end{equation}

Now let $A,B\in\MR$. Since
\[
A=\bigsqcup_{a\in A}\dR\uR(\{a\}),
\qquad
B=\bigsqcup_{b\in B}\dR\uR(\{b\}),
\]
and since multiplication in both $\MR$ and $\MS$ distributes over arbitrary
suprema, we obtain
\begin{align*}
\Bb_*(A\otimes_R B)
&=
\Bb_*\Big(
\Big(\bigsqcup_{a\in A}\dR\uR(\{a\})\Big)
\otimes_R
\Big(\bigsqcup_{b\in B}\dR\uR(\{b\})\Big)
\Big)\\
&=
\Bb_*\Big(
\bigsqcup_{a\in A,\ b\in B}
\left(\dR\uR(\{a\})\otimes_R\dR\uR(\{b\})\right)
\Big)\\
&=
\bigsqcup_{a\in A,\ b\in B}
\Bb_*\left(\dR\uR(\{a\})\otimes_R\dR\uR(\{b\})\right)
& (\text{by Proposition \ref{bond-map-sup}})\\
&=
\bigsqcup_{a\in A,\ b\in B}
\Bb_*\dR\uR(\{a\})\otimes_S\Bb_*\dR\uR(\{b\})
& (\text{by \eqref{Bb-preserves-multiplication-x-x'}})\\
&=
\Big(\bigsqcup_{a\in A}\Bb_*\dR\uR(\{a\})\Big)
\otimes_S
\Big(\bigsqcup_{b\in B}\Bb_*\dR\uR(\{b\})\Big)\\
&=
\Bb_*(A)\otimes_S\Bb_*(B). & (\text{by Proposition \ref{bond-map-sup}})
\end{align*}
Thus $\Bb_*$ preserves quantale multiplication.

{\bf Sufficiency.} Suppose that $\Bb_*$ preserves quantale multiplication. Let
$x,x'\in X$. Then it follows from \eqref{j-jx-jy} and Propositions
\ref{dRuR-nucleus}, \ref{MR-is-quantale} that
\[
\begin{aligned}
\Bb_*\dR\uR(\{xx'\})
&=
\Bb_*\left(\dR\uR(\{x\})\otimes_R\dR\uR(\{x'\})\right)\\
&=
\Bb_*\dR\uR(\{x\})\otimes_S\Bb_*\dR\uR(\{x'\}).
\end{aligned}
\]
Therefore,
\begin{align*}
\Bb[xx']
&=
\uS\Bb_*\dR\uR(\{xx'\})
&(\text{by Lemma \ref{bond-row-induced-map}})\\
&=
\uS\left(\Bb_*\dR\uR(\{x\})\otimes_S\Bb_*\dR\uR(\{x'\})\right)\\
&=
\uS\Bb_*\dR\uR(\{x\})\otimes_S\uS\Bb_*\dR\uR(\{x'\})
&(\text{by \eqref{P-otimes_S-Q-def}})\\
&=
\Bb[x]\otimes_S\Bb[x'].
&(\text{by Lemma \ref{bond-row-induced-map}})
\end{align*}
Thus $\Bb$ is multiplicative.
\end{proof}

\begin{rem} \label{multiplicative-bond-categorical-role}
Definition \ref{multiplicative-bond-def} formulates multiplicativity intrinsically: the row map
\[
\be_{\Bb}\colon X\to\Int(S)
\]
is required to be a semigroup homomorphism for the transported intent-side multiplication. Proposition \ref{multiplicative-bond-equivalence} shows that this row-wise condition is equivalent to preservation of multiplication by the induced map
\[
\Bb_*\colon\MR\to\MS.
\]
Since every bond already induces a sup-preserving map, a bond is multiplicative precisely when its induced map is a homomorphism of quantales. Thus multiplicative bonds provide the natural morphism-level lift of the object-level representation.
\end{rem}

\section{The category of residuated relations and the equivalence theorem}

We now define a category
\[
\ResRelMult,
\]
whose objects are residuated relations, and whose morphisms are multiplicative bonds. Its identity and composition are those of the underlying bonds in $\Bond$. The next two propositions show that these bonds are multiplicative, so this prescription is well-defined.

\begin{prop} \label{identity-multiplicative-bond}
Let $R\colon X\rto Y$ be a residuated relation. Then the relation $R$ itself is a multiplicative bond from $R$ to $R$. Moreover, the induced map
\[
R_*\colon \MR\to\MR
\]
is the identity map.
\end{prop}

\begin{proof}
As observed in the construction of $\Bond$, the relation $R$ is the identity
bond from $R$ to itself and $R_*=1_{\MR}$. For $x,x'\in X$, we have
\begin{align*}
R[x]\otimes_R R[x']
&=
\uR(\dR(R[x])*\dR(R[x']))
& (\text{by \eqref{P-otimes_S-Q-def}})\\
&=
\uR(\dR\uR(\{x\})*\dR\uR(\{x'\}))\\
&=
\uR\dR\uR(\dR\uR(\{x\})*\dR\uR(\{x'\}))\\
&=
\uR\dR\uR(\{x\}*\{x'\})
& (\text{by \eqref{j-jx-jy}})\\
&=
\uR(\{xx'\})\\
&=
R[xx'].
\end{align*}
Thus $R$ is multiplicative.
\end{proof}

\begin{prop} \label{composition-multiplicative-bonds}
The composite of two multiplicative bonds is again multiplicative.
\end{prop}

\begin{proof}
Let
\[
\Bb\colon R\to S
\qquad\text{and}\qquad
\Bc\colon S\to T
\]
be multiplicative bonds. By Proposition
\ref{multiplicative-bond-equivalence}, both $\Bb_*$ and $\Bc_*$ preserve
quantale multiplication. Hence so does $\Bc_*\circ\Bb_*$. Since the
composition in $\Bond$ satisfies
\[
(\Bc\circ\Bb)_*=\Bc_*\circ\Bb_*,
\]
Proposition \ref{multiplicative-bond-equivalence} again implies that
$\Bc\circ\Bb$ is multiplicative.
\end{proof}

Together with the associativity and unit laws in $\Bond$, these propositions
show that $\ResRelMult$ is a category. Forgetting the specified semigroup
structure on each domain and retaining the underlying bonds defines a
faithful functor
\[
\ResRelMult\to\Bond.
\]
It is not in general an inclusion of a subcategory: the same relation may
carry different semigroup structures on its domain and thereby define
distinct objects of $\ResRelMult$.

Next, we show that the assignment 
\[(\Bb\colon R\to S)\mapsto(\sM \Bb=\Bb_*\colon\MR\to\MS)\]
defines a functor
\[
\sM\colon \ResRelMult\to\Quant.
\]

\begin{prop} \label{M-is-functor}
$\sM\colon \ResRelMult\to\Quant$ is a well-defined functor.
\end{prop}

\begin{proof}
By Proposition \ref{MR-is-quantale}, $\sM$ is well-defined on objects. For
a multiplicative bond $\Bb\colon R\to S$, Proposition \ref{bond-map-sup} and
Proposition \ref{multiplicative-bond-equivalence} show that
\[
\Bb_*\colon\MR\to\MS
\]
is a quantale homomorphism. Thus $\sM$ is well-defined on morphisms.
Moreover, Proposition \ref{identity-multiplicative-bond} gives
\[
\sM(1_R)=R_*=1_{\MR},
\]
and the composition law in $\Bond$ gives
\[
\sM(\Bc\circ\Bb)
=
(\Bc\circ\Bb)_*
=
\Bc_*\circ\Bb_*
=
\sM\Bc\circ\sM\Bb.
\]
Therefore $\sM$ preserves identities and composition.
\end{proof}

\begin{rem} \label{residuated-relations-category-role}
The functor
\[
\sM\colon \ResRelMult\to\Quant
\]
records the passage from residuated relations to their associated concept
quantales. It sends a residuated relation to its concept quantale and a
multiplicative bond to the quantale homomorphism induced by that bond.
\end{rem}

\begin{thm} \label{main-equivalence-theorem}
$\sM\colon \ResRelMult\to\Quant$ is an equivalence of categories.
\end{thm}

\begin{proof}
For all objects $R,S$ of $\ResRelMult$, Corollary
\ref{bonds-sup-maps} and Proposition
\ref{multiplicative-bond-equivalence} yield a bijection
\[
\ResRelMult(R,S)\to\Quant(\MR,\MS),
\qquad
\Bb\mapsto\Bb_*.
\]
Thus $\sM$ is full and faithful. Theorem
\ref{sup-inf-dense-criterion} shows that every quantale is isomorphic to
$\MR$ for some object $R$ of $\ResRelMult$, so $\sM$ is essentially
surjective. Hence $\sM$ is an equivalence of categories.
\end{proof}



\begin{rem} \label{equivalence-theorem-meaning}
The equivalence above is the natural morphism-level completion of the
object-level representation theorem. Theorem
\ref{sup-inf-dense-criterion} says that every quantale can be represented as
the concept lattice of a residuated relation, and characterizes such
representations by sup-dense and inf-dense maps. Theorem
\ref{main-equivalence-theorem} identifies the homomorphisms between the
represented quantales precisely with the maps induced by multiplicative bonds.
Thus residuated relations and multiplicative bonds provide a categorical
presentation of quantales.
\end{rem}

\section{The semigroup-object interpretation}
\label{section-semigroup-objects}

The category $\Bond$ carries more structure than its equivalence with $\Sup$
alone records. Mori established a lattice tensor on $\Bond$, denoted here by
$\boxtimes$, and a strong monoidal equivalence
\[
(\Bond,\boxtimes)
\simeq_{\otimes}
(\Sup,\boxtimes),
\]
where the tensor on the right is the standard tensor product of complete
sup-lattices \cite{Mori2008}; see also \cite[Theorem~27]{Tull2020}. With the non-unital convention adopted in this paper, a quantale is a \emph{semigroup object} \cite{MacLane1998} in $(\Sup,\boxtimes)$, while a unital
quantale is a monoid object. For a monoidal category $(\CC,\boxtimes)$, we
write $\SGrp(\CC,\boxtimes)$ for its category of semigroup objects.

Via the canonical isomorphism of categories between $\Bond$ and
$\BB(\Rel)$ (see Remark \ref{Bond-BRel}), we transport Mori's
monoidal structure from $\Bond$ to $\BB(\Rel)$. Equipped with the transported
tensor $\boxtimes$, this isomorphism is strong monoidal by construction.

\begin{prop}
\label{residuated-relations-semigroup-objects}
There are equivalences of categories
\[
\ResRelMult
\simeq
\Quant
\simeq
\SGrp(\Bond,\boxtimes)
\simeq
\SGrp(\BB(\Rel),\boxtimes).
\]
\end{prop}

\begin{proof}
Theorem~\ref{main-equivalence-theorem} gives the first equivalence. The strong
monoidal equivalence above induces an equivalence between the corresponding
categories of semigroup objects. The semigroup objects of $(\Sup,\boxtimes)$ are precisely
the non-unital quantales. Hence
\[
\SGrp(\Bond,\boxtimes)
\simeq
\SGrp(\Sup,\boxtimes)
=
\Quant.
\]
Finally, the transported strong monoidal isomorphism
\[
(\Bond,\boxtimes)
\cong_{\otimes}
(\BB(\Rel),\boxtimes)
\]
induces an equivalence
\[
\SGrp(\Bond,\boxtimes)
\simeq
\SGrp(\BB(\Rel),\boxtimes),
\]
which gives the final equivalence.
\end{proof}

This proposition is a categorical interpretation of the main theorem, rather
than a replacement for its relation-level proof. An object of
$\ResRelMult$ is a concrete residuated presentation
\[
R\colon X\rto Y
\]
whose object set $X$ carries a semigroup structure; its attribute set $Y$ is
only a set. The residual equations use the multiplication on $X$ to make the
Galois closure
\[
\sj_R=\dR\uR
\]
into a quantic nucleus, and hence to make the concept lattice $\MR$ into a
quantale. The row condition on a multiplicative bond then says exactly that
the induced sup-lattice map preserves this descended multiplication.

By contrast, a semigroup object of $(\Bond,\boxtimes)$ is an internal
multiplication in the bond category. The proposition says that, up to the
displayed categorical equivalences, these internal semigroup objects are
represented by the concrete residuated presentations described in this paper.
It does not say that an arbitrary internal semigroup object reconstructs, on
its original
object and attribute sets, a semigroup operation and residual maps of the
displayed form. This distinction is the value of the relation-level
description: the residual equations and the row criterion supply explicit
object--attribute data whose Galois quotient realizes the abstract internal
multiplication.

\section{Pseudogroup completions and \'{e}tale groupoids}
\label{section-pseudogroup-completions}

Let $S$ be an \emph{abstract complete pseudogroup}
(see \cite[Definitions~2.8 and~2.9]{Resende2007}), namely a complete and
infinitely distributive \emph{inverse semigroup} \cite{Clifford1961,Paterson1999}. For $s\in S$, write $s^{-1}$
for its unique inverse. The natural order on $S$ is given by
\[
s\leq t
\iff
s=et\quad\text{for some idempotent }e\in S.
\]
A subset of $S$ is \emph{compatible} if $s^{-1}t$ and $st^{-1}$ are
idempotent for all $s,t$ in the subset. Completeness means that
every compatible subset, including the empty subset, has a supremum; in
particular, $S$ is an inverse monoid. Infinite distributivity means that
multiplication distributes over these suprema. The construction below realizes
the enveloping quantale of $S$ as a concept quantale, thereby giving an FCA
realization of the pseudogroup completion (see
\cite[Definition~3.22 and Lemma~3.23]{Resende2007}). Let
$\mathcal L^\vee(S)$ be the set of compatibly closed ideals of $S$, namely
the subsets that are downward closed in the natural order and closed under
compatible suprema. For
$A\subseteq S$, put
\[
\operatorname{cl}_S(A)
=
\bigcap\{U\in\mathcal L^\vee(S)\mid A\subseteq U\}.
\]
For $U,V\in\mathcal L^\vee(S)$, write
\[
U\star V=\operatorname{cl}_S(U*V),
\qquad
U^\dagger=\{u^{-1}\mid u\in U\}.
\]
Thus the dagger on $\mathcal L^\vee(S)$ denotes the pointwise extension of
the inverse operation on $S$.

Recall that an \emph{involution} on a quantale $Q$ is a sup-preserving map
$a\mapsto a^\dagger$ such that
$a^{\dagger\dagger}=a$ and
$(a*b)^\dagger=b^\dagger*a^\dagger$ for all $a,b\in Q$. A
\emph{support} on a unital involutive quantale $Q$, with unit $e$, is a
sup-preserving map $\varsigma\colon Q\to Q$ satisfying
\[
\varsigma(a)\leq e,
\qquad
\varsigma(a)\leq a*a^\dagger,
\qquad
a\leq\varsigma(a)*a.
\]
A \emph{frame} is a complete lattice in which finite meets distribute over
arbitrary suprema. A \emph{supported quantal frame} is a unital involutive
quantale whose underlying complete lattice is a frame and which is equipped
with a support. An \emph{inverse quantal frame} is a supported quantal frame
whose top element is the supremum of its partial units, where $u$ is a partial unit when
$u*u^\dagger\leq e$ and $u^\dagger*u\leq e$
(see \cite[Definition~4.10]{Resende2007}).

Then $\mathcal L^\vee(S)$ is a unital involutive quantal frame, with unit
$\downarrow 1_S$ and multiplication $\star$
(see \cite[Lemma~3.23]{Resende2007}).

\begin{defn}
\label{pseudogroup-membership-context}
With $S$ as above, the \emph{pseudogroup membership context} of $S$ is
\[
R_S\colon S\rto\mathcal L^\vee(S),
\qquad
sR_SU\iff s\in U,
\]
with object semigroup $S$ and attribute set $\mathcal L^\vee(S)$.
\end{defn}

This is the residual-stable membership construction of
Example~\ref{example-residual-stable-tests} with
$\mathcal T=\mathcal L^\vee(S)$; the next proposition verifies the
required residual stability in this case.

\begin{prop}
\label{pseudogroup-membership-is-residuated}
The relation $R_S$ is residuated. Its residual operations are
\[
U\lda t=\{s\in S\mid st\in U\},
\qquad
s\rda U=\{t\in S\mid st\in U\}.
\]
Moreover,
\[
sR_SU\iff s^{-1}R_SU^\dagger .
\]
\end{prop}

\begin{proof}
For $U\in\mathcal L^\vee(S)$ and $t\in S$, the set
$U\lda t$ is downward closed. If $E\subseteq U\lda t$ is compatible, then
$Et$ is compatible and
\[
(\bigvee E)t=\bigvee(Et)\in U.
\]
Thus $U\lda t$ is compatibly closed. Similarly, if
$E\subseteq s\rda U$ is compatible, then $sE$ is compatible and
\[
s(\bigvee E)=\bigvee(sE)\in U,
\]
so $s\rda U\in\mathcal L^\vee(S)$. The defining equivalences are
\[
(st)R_SU
\Longleftrightarrow st\in U
\Longleftrightarrow s\in U\lda t
\Longleftrightarrow sR_S(U\lda t)
\]
and, symmetrically, $(st)R_SU\Longleftrightarrow tR_S(s\rda U)$.
Finally, $U^\dagger\in\mathcal L^\vee(S)$ and
$s\in U$ if and only if $s^{-1}\in U^\dagger$.
\end{proof}

\begin{prop}
\label{pseudogroup-concept-quantale}
For every $A\subseteq S$,
\[
\dR_S\uR_S(A)=\operatorname{cl}_S(A).
\]
Consequently, the extents of $R_S$ are precisely the compatibly closed
ideals, and
\[
\MR_S\cong\mathcal L^\vee(S)
\]
as unital involutive quantales.
\end{prop}

\begin{proof}
One has
\[
\uR_S(A)=\{U\in\mathcal L^\vee(S)\mid A\subseteq U\};
\]
hence its lower transform is exactly $\operatorname{cl}_S(A)$. Thus the
extents are the members of $\mathcal L^\vee(S)$. For closed ideals $U,V$,
\[
U\otimes_{R_S}V
=
\dR_S\uR_S(U*V)
=
\operatorname{cl}_S(U*V)
=
U\star V.
\]
The unit $\downarrow 1_S$ and the involution $U\mapsto U^\dagger$ are
therefore those of $\mathcal L^\vee(S)$, which proves the assertion.
\end{proof}

\begin{cor}
\label{pseudogroup-etale-groupoid}
For every abstract complete pseudogroup $S$, the concept quantale
$\MR_S$ is an inverse quantal frame. Hence there is a localic \'{e}tale
groupoid $G_S$ for which
\[
\MR_S\cong\mathcal O(G_S)
\]
as unital involutive quantales, where $\mathcal O(G_S)$ denotes the quantale
of opens of the arrow locale of $G_S$.
\end{cor}

\begin{proof}
Proposition~\ref{pseudogroup-concept-quantale} identifies $\MR_S$ with
$\mathcal L^\vee(S)$. The latter is an inverse quantal frame
(see \cite[Definition~3.16, Lemma~3.23 and Definition~4.10]{Resende2007}), and the groupoid
statement follows from the inverse-quantal-frame reconstruction and recovery
results (see \cite[Theorems~4.19 and~5.11]{Resende2007}).
\end{proof}

Here a homomorphism of abstract complete pseudogroups means a monoid
homomorphism that preserves suprema of compatible subsets. Let
$f\colon S\to T$ be such a homomorphism.
Its universal extension
\[
\mathcal L^\vee(f)\colon\mathcal L^\vee(S)\to\mathcal L^\vee(T),
\qquad
U\mapsto\operatorname{cl}_T(f[U])
\]
is a unital involutive quantale homomorphism
(see \cite[Theorem~3.25 and Corollary~3.26]{Resende2007}). The associated multiplicative bond
\[
\Bb_f\colon R_S\to R_T
\]
has the particularly transparent form
\[
s\,\Bb_f\,V
\iff
f(s)\in V,
\qquad
(s\in S,\;V\in\mathcal L^\vee(T)).
\]
Indeed,
\[
\dR_S\uR_S(\{s\})=\downarrow s
\qquad\text{and}\qquad
\mathcal L^\vee(f)(\downarrow s)=\downarrow f(s).
\]
Since $V$ is downward closed, the bond obtained from
$\mathcal L^\vee(f)$ by Proposition~\ref{sup-map-bond} has precisely the
displayed incidence relation. It is multiplicative by
Proposition~\ref{multiplicative-bond-equivalence}.


\begin{appendices}

\section{Examples of residuated relations}
\label{appendix-examples}

The examples below illustrate four recurrent sources of residuated relations:
order comparisons, membership in residual-stable families of subsets,
quantale-valued observations, and equations determined by multiplicative
invariants. In each case we display the residual maps and identify the
resulting concept quantale. Some examples are special cases of more general
ones, but are retained because the residual operations have familiar
interpretations in their respective settings.

\subsection{Order-induced relations}

\begin{exmp}[Dense ordered presentations]
\label{example-sup-inf-dense-ordered-presentations}
Let $Q$ be a quantale, and let $A\subseteq Q$ be a subsemigroup that is both
sup-dense and inf-dense in $Q$. Suppose that $A$, with the order inherited
from $Q$, is a residuated ordered semigroup. Then the restricted order
relation
\[
{\leq}\ \colon A\rto A
\]
is residuated. The inclusion $A\to Q$ is at once a sup-dense semigroup
homomorphism and an inf-dense map, so
Theorem~\ref{sup-inf-dense-criterion} gives
\[
\sM(A,A,\leq)\cong Q
\]
as quantales.

For example, take $Q=([0,1],\wedge)$ and
$A=\bbQ\cap[0,1]$. The residual in $A$ is
\[
a\ra b=
\begin{cases}
1 & \text{if }a\leq b,\\
b & \text{if }b<a.
\end{cases}
\]
Thus the rational order context presents the quantale
$([0,1],\wedge)$.

As a non-idempotent example, consider the Lawvere quantale
$Q=([0,\infty],\geq,+)$ \cite{Lawvere1973}, ordered oppositely to the usual
order, and put $A=\bbQ_{\geq0}\cup\{\infty\}$. This is a sup-dense and
inf-dense subsemigroup. Its residual is truncated subtraction:
\[
b\ominus a=
\begin{cases}
0 & \text{if }b\leq a,\\
b-a & \text{if }a<b<\infty,\\
\infty & \text{if }b=\infty\text{ and }a<\infty,
\end{cases}
\]
where the displayed inequalities use the usual order. Hence
$(A,A,\geq)$ presents $([0,\infty],\geq,+)$.
\end{exmp}

\begin{exmp}[A frame from a meet-semilattice basis]
\label{example-locales-basis-tests}
Let $L$ be a frame and let $B\subseteq L$ be a meet-semilattice basis \cite{Johnstone1986}.
Regard $B$ as a semigroup under finite meet, and define
\[
R\colon B\rto L,
\qquad
bRu
\iff
b\leq u.
\]
If $\ra$ denotes the Heyting implication of $L$, put
\[
u\lda b'=b'\ra u,
\qquad
b\rda u=b\ra u.
\]
Then
\[
b\wedge b'\leq u
\iff
b\leq b'\ra u
\iff
b'\leq b\ra u,
\]
so $R$ is residuated. The inclusion $B\to L$ is sup-dense, while
$1_L\colon L\to L$ is inf-dense. Therefore
\[
\sM R\cong L
\]
as quantales. Thus a frame can be presented by its basic elements as objects
and all its elements as attributes.
\end{exmp}

\begin{exmp}[Quantic nuclei as restricted attributes]
\label{example-nuclei-test-restrictions}
Let $Q$ be a quantale, let $\sj$ be a quantic nucleus on $Q$, and let
$f\colon X\to Q$ be a sup-dense semigroup homomorphism. Take the fixed-point
set $Q_{\sj}$ as the attribute set and define
\[
R_{\sj}\colon X\rto Q_{\sj},
\qquad
xR_{\sj}y
\iff
f(x)\leq y.
\]
For $y\in Q_{\sj}$ and $x,x'\in X$, put
\[
y\lda x'=y\lda f(x'),
\qquad
x\rda y=f(x)\rda y,
\]
where the operations on the right are the residuals in $Q$. These elements
belong to $Q_{\sj}$. Indeed, if $c=y\lda f(x')$, then
\[
\sj(c)*f(x')
\leq
\sj(c)*\sj f(x')
\leq
\sj(c*f(x'))
\leq
\sj(y)
=
y,
\]
and hence $\sj(c)\leq c$; the reverse inequality is automatic. The other
residual is treated similarly. It follows that $R_{\sj}$ is residuated.

For every $A\subseteq X$, its Galois closure is
\[
\dR_{\sj}\uR_{\sj}(A)
=
\left\{
x\in X
\mathrel{\Big|}
f(x)\leq\sj\Big(\bv f[A]\Big)
\right\}.
\]
The map $x\mapsto\sj f(x)$ is a sup-dense semigroup homomorphism
$X\to Q_{\sj}$, and the identity map of $Q_{\sj}$ is inf-dense. Hence
\[
\sM R_{\sj}\cong Q_{\sj}
\]
as quantales. This realizes a quantale quotient by retaining the original
objects and restricting the attributes to the closed elements of the
nucleus.
\end{exmp}

\subsection{Membership relations}

\begin{exmp}[The powerset membership context]
\label{example-powerset-context}
\label{example-groups-elements-subsets}
Let $X$ be a semigroup. The membership relation
\[
R_X\colon X\rto\PX,
\qquad
xR_XA\iff x\in A
\]
is residuated, with
\[
A\lda x'=\{x\in X\mid xx'\in A\},
\qquad
x\rda A=\{x'\in X\mid xx'\in A\}.
\]
Indeed, $(xx')R_XA$ holds precisely when
$x\in A\lda x'$, equivalently when $x'\in x\rda A$. Its Galois closure
is the identity on $\PX$, since every subset is itself an attribute. Hence
\[
\MR_X\cong\PX
\]
as quantales. This is the unclosed limiting case of the membership
construction.
\end{exmp}

\begin{exmp}[Quotient attributes from a semigroup homomorphism]
\label{example-semigroup-quotient-tests}
Let
\[
h\colon X\to M
\]
be a surjective semigroup homomorphism. Put
\[
\mathcal T_h=\{h^{-1}(B)\mid B\subseteq M\}.
\]
For $B\subseteq M$ and $x,x'\in X$, the point residuals are
\[
h^{-1}(B)\lda x'
=
h^{-1}(\{m\in M\mid m h(x')\in B\})
\]
and
\[
x\rda h^{-1}(B)
=
h^{-1}(\{m\in M\mid h(x)m\in B\}).
\]
Thus the displayed residuals belong to $\mathcal T_h$, and the membership
relation
\[
R_h\colon X\rto\mathcal T_h,
\qquad
xR_hh^{-1}(B)
\iff
h(x)\in B
\]
is residuated. Its Galois closure is
\[
\dR_h\uR_h(A)=h^{-1}(h[A])
\]
for every $A\subseteq X$. The extents are therefore exactly the saturated
subsets $h^{-1}(B)$, and
\[
\MR_h\cong\mathsf P M
\]
as quantales. Hence a surjective semigroup homomorphism produces a
formal-concept presentation of the powerset quantale of its quotient
semigroup.
\end{exmp}

\begin{exmp}[Residual-stable families of attributes]
\label{example-residual-stable-tests}
\label{example-residual-stable-membership-tests}
Let $X$ be a semigroup and let $\mathcal T$ be a family of subsets of $X$
that is stable under the point residuals
\[
T\lda x=\{u\in X\mid ux\in T\}\in\mathcal T,
\qquad
x\rda T=\{u\in X\mid xu\in T\}\in\mathcal T
\]
for all $T\in\mathcal T$ and $x\in X$. Then the membership relation
\[
R_{\mathcal T}\colon X\rto\mathcal T,
\qquad
xR_{\mathcal T}T\iff x\in T
\]
is residuated. Its Galois closure
\[
\sj_{\mathcal T}:=\dR_{\mathcal T}\uR_{\mathcal T}
\]
is
\[
\sj_{\mathcal T}(A)
=
\bigcap\{T\in\mathcal T\mid A\subseteq T\},
\]
where the empty intersection is $X$. Proposition~\ref{dRuR-nucleus}
therefore says that $\sj_{\mathcal T}$ is a quantic nucleus, and its fixed
points form a quantale.

This construction allows one to choose a proper family of attributes rather
than all subsets of $X$. Residual stability ensures that
the resulting closure respects multiplication. The pseudogroup membership
context of Section~\ref{section-pseudogroup-completions} is an instance of
this mechanism.
\end{exmp}

\begin{rem}[Formal languages and trace specifications]
\label{example-languages-traces}
\label{example-regular-language-tests}
For the free monoid $X=\Sigma^*$, the residuals in the powerset membership
context are the familiar right and left language quotients
\cite{Eilenberg1974}:
\[
L\lda v
=
L/v
=
\{u\in\Sigma^*\mid uv\in L\},
\qquad
u\rda L
=
u\backslash L
=
\{v\in\Sigma^*\mid uv\in L\}.
\]
Choosing all languages gives the language quantale
$\mathsf P(\Sigma^*)$. Fixing a surjective monoid homomorphism
$h\colon\Sigma^*\to M$ instead and using only inverse images
of subsets of $M$ gives Example~\ref{example-semigroup-quotient-tests}, with
closure $A\mapsto h^{-1}(h[A])$ and concept quantale $\mathsf P M$. By
contrast, when $\Sigma$ is finite, allowing all regular languages still
gives the identity closure: for $w\notin A$, the regular language
$\Sigma^*\setminus\{w\}$ contains $A$ but not $w$. The same quotient
formulas apply to a semigroup of program traces, as in trace semantics
\cite{Hoare1985}, where $L/v$ describes the prefixes that satisfy $L$ after
the continuation $v$.
\end{rem}

\begin{exmp}[Ideal and radical-ideal attributes]
\label{example-ideals-radical-ideals}
Let $A$ be a commutative ring, regarded as a multiplicative semigroup.
First consider the membership relation
\[
R_{\mathrm{Idl}}\colon A\rto\operatorname{Idl}(A),
\qquad
aR_{\mathrm{Idl}}I
\iff
a\in I,
\]
where the attribute set is the set of ideals. For
$I\in\operatorname{Idl}(A)$ and $b\in A$, both residual attributes are the
\emph{colon ideal of $I$ by $b$} \cite{Atiyah1969},
\[
I\lda b
=
(I:b)
=
\{a\in A\mid ab\in I\},
\qquad
b\rda I=(I:b).
\]
Thus $R_{\mathrm{Idl}}$ is residuated. Its induced closure sends
$S\subseteq A$ to the ideal generated by $S$:
\[
\dR_{\mathrm{Idl}}\uR_{\mathrm{Idl}}(S)=(S).
\]
Consequently,
\[
\sM R_{\mathrm{Idl}}\cong\operatorname{Idl}(A)
\]
as quantales, with multiplication given by ideal product.

Now restrict the attributes to radical ideals. If $I$ is radical, then
$(I:b)$ is radical. Indeed, if $c^n\in(I:b)$, then $c^n b\in I$, and
\[
(cb)^n=(c^n b)b^{n-1}\in I.
\]
It follows that $cb\in I$, and hence $c\in(I:b)$. Therefore the same colon
ideals give residual attributes for
\[
R_{\mathrm{rad}}\colon A\rto\operatorname{RIdl}(A),
\qquad
aR_{\mathrm{rad}}I
\iff
a\in I.
\]
Its closure is
\[
\dR_{\mathrm{rad}}\uR_{\mathrm{rad}}(S)=\sqrt{(S)},
\]
and hence
\[
\sM R_{\mathrm{rad}}\cong\operatorname{RIdl}(A)
\]
as quantales, with multiplication $I\otimes J=\sqrt{IJ}$.
\end{exmp}

\begin{rem}[Affine-variety form]
\label{example-ideals-radical-ideals-varieties}
Let $A=K[x_1,\ldots,x_n]$, where $K$ is algebraically closed, and take the
affine algebraic subsets of $K^n$ as attributes. Define
\[
fR_{\mathrm{Zar}}V
\iff
f|_V=0.
\]
For $g\in A$, put
\[
V\lda g
=
V(I(V):g),
\qquad
g\rda V
=
V(I(V):g).
\]
The ideal $(I(V):g)$ is radical by the preceding example. The Hilbert
Nullstellensatz \cite{Hartshorne1977} therefore gives
\[
fR_{\mathrm{Zar}}(V\lda g)
\iff
f\in(I(V):g)
\iff
fg\in I(V)
\iff
(fg)R_{\mathrm{Zar}}V,
\]
and similarly for the other residual. The induced closure is
$S\mapsto\sqrt{(S)}$. Thus this relation presents the radical-ideal quantale,
or equivalently the lattice of affine algebraic sets in the opposite order,
with multiplication corresponding to union of algebraic sets.
\end{rem}

\begin{exmp}[Closed-set attributes in a topological semigroup]
\label{example-closed-set-tests}
Let $X$ be a \emph{topological semigroup} \cite{Carruth1983}, and let $\CX$ be the set of closed
subsets of $X$. The membership relation
\[
R_{\mathrm{cl}}\colon X\rto\CX,
\qquad
xR_{\mathrm{cl}}F
\iff
x\in F
\]
is residuated. Indeed, for $F\in\CX$ and $a,b\in X$, the residual attributes
\[
F\lda b
=
\{a\in X\mid ab\in F\},
\qquad
a\rda F
=
\{b\in X\mid ab\in F\}
\]
are closed, since they are inverse images of $F$ under continuous right and
left translations. Moreover,
\[
\dR_{\mathrm{cl}}\uR_{\mathrm{cl}}(A)=\overline A
\]
for all $A\subseteq X$. Consequently,
\[
\sM R_{\mathrm{cl}}\cong\CX
\]
as quantales, with multiplication
\[
F\otimes G=\overline{F*G}.
\]
Thus ordinary topological closure is obtained directly as the Galois closure
of a natural residuated membership relation.
\end{exmp}

\begin{exmp}[Closed linear subspaces of operator algebras]
\label{example-operator-algebras}
Let $A$ be a $C^*$-algebra, regarded as a multiplicative semigroup, and let
$\operatorname{Max}A$ be the set of norm-closed complex linear subspaces of
$A$ \cite{Kruml2004}. Define
\[
R_{\operatorname{Max}}\colon A\rto\operatorname{Max}A,
\qquad
aR_{\operatorname{Max}}M
\iff
a\in M.
\]
For $N\in\operatorname{Max}A$ and $a,b\in A$, put
\[
N\lda b
=
\{a\in A\mid ab\in N\},
\qquad
a\rda N
=
\{b\in A\mid ab\in N\}.
\]
These are norm-closed complex linear subspaces, being inverse images of $N$
under bounded linear maps. Hence $R_{\operatorname{Max}}$ is residuated, and
its Galois closure is
\[
\dR_{\operatorname{Max}}\uR_{\operatorname{Max}}(S)
=
\overline{\operatorname{span}(S)}.
\]
It follows that
\[
\sM R_{\operatorname{Max}}\cong\operatorname{Max}A
\]
as quantales, with multiplication
\[
M\otimes N
=
\overline{\operatorname{span}(\{mn\mid m\in M,\ n\in N\})}.
\]

The same argument applies to a von Neumann algebra with weak-* closed linear
subspaces as attributes \cite{Egger2010}; see \cite{Pelletier1997a} for the
weak-operator version. Since multiplication by a
fixed element is weak-*
continuous, the residual attributes are weak-* closed, and the induced
closure is weak-* closed linear span. If $A$ is a commutative $C^*$-algebra,
one may instead use norm-closed ideals as attributes; the resulting concept
quantale is the quantale of closed ideals, with multiplication
$I\otimes J=\overline{IJ}$. For a general noncommutative $C^*$-algebra,
closed two-sided ideals need not be stable under both point residuals.
\end{exmp}

\subsection{Observation-induced relations}

\begin{exmp}[Quantale-valued observations]
\label{example-quantale-valued-observations}
Let $X$ be a semigroup, let $Q$ be a quantale, and let
\[
\Phi\colon X\to Q
\]
be a semigroup homomorphism. For the interpretation of quantales as
algebras of observations, see \cite{Resende2000}. Define
\[
R_{\Phi}\colon X\rto Q,
\qquad
xR_{\Phi}q
\iff
\Phi(x)\leq q.
\]
The quantale residuals give
\[
q\lda x'=q\lda\Phi(x'),
\qquad
x\rda q=\Phi(x)\rda q,
\]
and hence $R_{\Phi}$ is residuated. For every $A\subseteq X$,
\[
\dR_{\Phi}\uR_{\Phi}(A)
=
\left\{
x\in X
\mathrel{\Big|}
\Phi(x)\leq\bv_{a\in A}\Phi(a)
\right\}.
\]
Consequently, $\sM R_{\Phi}$ is isomorphic to the subquantale
\[
Q_{\Phi}
=
\left\{
\bv_{a\in A}\Phi(a)
\mathrel{\Big|}
A\subseteq X
\right\}
\]
of $Q$. Thus every multiplicative representation of a semigroup in a
quantale canonically yields a residuated relation.
\end{exmp}

\begin{rem}[Relational semantics and group actions]
\label{example-programs-relational-specifications}
\label{example-group-actions-relational-tests}
Let $S$ be a set and take $Q=\mathsf P(S\times S)$ with relational
composition in the convention
\[
P\circ Q
=
\{(s,t)\mid \text{there exists }u\in S\text{ with }(s,u)\in P
\text{ and }(u,t)\in Q\}.
\]
This is the standard quantale of binary relations; see \cite{Freyd1990}
for the relational background. A semigroup homomorphism
\[
\Phi\colon X\to\mathsf P(S\times S)
\]
may be viewed as a compositional relational semantics for the elements of
$X$ \cite{Abramsky1993}. The preceding relation becomes
\[
xR_{\Phi}T
\iff
\Phi(x)\subseteq T,
\]
and its concept quantale is the subquantale generated by the relational
behaviours $\Phi(x)$. In particular, a right action of a group $G$ on $S$
gives such a homomorphism by sending $g$ to the graph of
$s\mapsto s\cdot g$.
\end{rem}

\begin{rem}[Closed set-valued observations]
\label{example-set-valued-closed-tests}
Let $X$ be a semigroup, let $Y$ be a topological semigroup, and let
$\Phi\colon X\to\mathsf P Y$ be a semigroup homomorphism. For general
background on set-valued maps, see \cite{Aubin1990}. Let $\CY$ be the set of
closed subsets of $Y$, and define
\[
R_{\Phi}^{\mathrm{cl}}\colon X\rto\CY,
\qquad
xR_{\Phi}^{\mathrm{cl}}C
\iff
\Phi(x)\subseteq C
\]
for $x\in X$ and $C\in\CY$. For $C\in\CY$ and $x,x'\in X$, put
\[
C\lda x'
=
\{y\in Y\mid y\Phi(x')\subseteq C\},
\qquad
x\rda C
=
\{y\in Y\mid\Phi(x)y\subseteq C\}.
\]
These two residual attributes are closed, being intersections of inverse
images of $C$ under continuous right and left translations, respectively.
Since $\Phi(xx')=\Phi(x)*\Phi(x')$, they satisfy
\[
(xx')R_{\Phi}^{\mathrm{cl}}C
\iff
xR_{\Phi}^{\mathrm{cl}}(C\lda x')
\iff
x'R_{\Phi}^{\mathrm{cl}}(x\rda C).
\]
Thus $R_{\Phi}^{\mathrm{cl}}$ is residuated. Writing
$\Phi[A]=\bigcup_{a\in A}\Phi(a)$, its Galois closure is
\[
\sj_{R_{\Phi}^{\mathrm{cl}}}(A)
=
\{x\in X\mid\Phi(x)\subseteq\overline{\Phi[A]}\}.
\]
Applying $\sj_{R_{\Phi}^{\mathrm{cl}}}$ to $A$ does not change
$\overline{\Phi[A]}$. Hence the map
\[
\Fix\bigl(\sj_{R_{\Phi}^{\mathrm{cl}}}\bigr)\longrightarrow\CY,
\qquad
A\longmapsto\overline{\Phi[A]},
\]
identifies the fixed-point quantale with the subquantale
\[
\{\overline{\Phi[A]}\mid A\subseteq X\}\subseteq\CY,
\]
whose suprema are closures of unions and whose multiplication is
$F\otimes G=\overline{F*G}$.
\end{rem}

\subsection{Multiplicative invariants}

\begin{exmp}[Group-valued multiplicative invariants]
\label{example-multiplicative-invariants}
Let $X$ be a semigroup, let $G$ be a group, let
$h\colon X\to G$ be a semigroup homomorphism, and fix $q\in G$. Define
\[
R_q\colon X\rto G,
\qquad
xR_qg
\iff
h(x)g=q.
\]
For $g\in G$ and $x,x'\in X$, put
\[
g\lda x'=h(x')g,
\qquad
x\rda g=gq^{-1}h(x)q.
\]
A direct calculation gives
\[
(xx')R_qg
\iff
xR_q(g\lda x')
\iff
x'R_q(x\rda g),
\]
so $R_q$ is residuated. If $G$ is commutative, the second formula simplifies
to $x\rda g=gh(x)$.

For $a\in h[X]$, let $F_a=h^{-1}(a)$. Then
\[
\uR_q(F_a)=\{a^{-1}q\},
\qquad
\dR_q\uR_q(F_a)=F_a,
\]
so every nonempty fibre of $h$ is an extent. A subset meeting two distinct
fibres has no common attribute and therefore closes to $X$. Moreover, for
$a,b\in h[X]$,
\[
F_a\otimes_{R_q}F_b=F_{ab}.
\]
Thus the concept quantale retains the multiplication of the invariant values
while identifying elements in the same fibre.
\end{exmp}

\begin{exmp}[Determinant tests and singular collapse]
\label{example-matrix-determinant-context}
Let $K$ be a field, let $n\geq 1$, regard $M_n(K)$ as a multiplicative
semigroup, and take $K$ as the attribute set. Fix
$q\in K^*=K\setminus\{0\}$, and define
\[
\phi_q\colon M_n(K)\rto K,
\qquad
A\phi_q\lambda
\iff
\det(A)\lambda=q.
\]
On $GL_n(K)$, this is the preceding construction for the group-valued
invariant $\det\colon GL_n(K)\to K^*$. The present context also includes
singular matrices and the zero scalar.

For $A,A'\in M_n(K)$ and $\lambda\in K$, put
\[
\lambda\lda A'=\det(A')\lambda,
\qquad
A\rda\lambda=\lambda\det(A).
\]
Since $K$ is commutative and the determinant is multiplicative
\cite{Artin2010},
\[
(AA')\phi_q\lambda
\iff
A\phi_q(\lambda\lda A')
\iff
A'\phi_q(A\rda\lambda).
\]
Hence $\phi_q$ is a residuated relation.

For $\mu\in K$, write
\[
D_\mu=\{A\in M_n(K)\mid\det(A)=\mu\}.
\]
If $\mu\in K^*$, then
\[
\uphi_q(D_\mu)=\{q\mu^{-1}\}
\qquad\text{and}\qquad
\dphi_q\uphi_q(D_\mu)=D_\mu.
\]
On the other hand,
\[
\uphi_q(D_0)=\varnothing
\qquad\text{and}\qquad
\dphi_q\uphi_q(D_0)=M_n(K),
\]
because no scalar satisfies $0\lambda=q$. More generally, for every
$S\subseteq M_n(K)$,
\[
\dphi_q\uphi_q(S)
=
\begin{cases}
\varnothing
& \text{if }S=\varnothing,\\
D_\mu
& \text{if }S\neq\varnothing\text{ and }\det(A)=\mu\in K^*
  \text{ for all }A\in S,\\
M_n(K)
& \text{otherwise}.
\end{cases}
\]
Indeed, a nonempty family of matrices has a common scalar attribute precisely
when all its members have the same nonzero determinant. Consequently, the
extents of $\phi_q$ are exactly
\[
\varnothing,
\qquad
D_\mu\quad(\mu\in K^*),
\qquad
M_n(K).
\]

For $\mu,\nu\in K^*$, one has
\[
D_\mu*D_\nu=D_{\mu\nu}.
\]
The inclusion $D_\mu*D_\nu\subseteq D_{\mu\nu}$ follows from the
multiplicativity of the determinant. Conversely, for $C\in D_{\mu\nu}$,
choose $A\in D_\mu$ and put $B=A^{-1}C$; then $B\in D_\nu$ and $C=AB$.
Therefore
\[
D_\mu\otimes_{\phi_q}D_\nu=D_{\mu\nu}.
\]
Moreover, $\varnothing$ is multiplicatively absorbing, while
\[
M_n(K)\otimes_{\phi_q}E
=
E\otimes_{\phi_q}M_n(K)
=
M_n(K)
\]
for every nonempty extent $E$. Thus the nonzero determinant fibres reproduce
the group multiplication of $K^*$, whereas every singular matrix is
indistinguishable from the top extent under the chosen nonzero scalar
attributes.
\end{exmp}

\end{appendices}

\section*{Acknowledgements}

The first named author acknowledges support from the National Natural Science Foundation of China (Grant No. 12671561). The authors would like to thank Professor Hongliang Lai for helpful discussions.

\section*{Declaration of generative AI and AI-assisted technologies
in the manuscript preparation process}

During the preparation of this work, the authors used Codex in
ChatGPT Work (OpenAI) to assist with manuscript organization,
language editing, consistency checking, and submission-readiness
review. After using this tool, the authors reviewed and edited the
content as needed and take full responsibility for the content of
the published article.



\bibliographystyle{abbrv}
\bibliography{lili}

\end{document}